\documentclass[reqno,12pt]{amsart}
\usepackage[english]{babel}
\usepackage{amsmath,amsfonts,amsthm,amssymb,amsbsy,upref,color,graphicx,hyperref,enumerate,comment}
\usepackage[a4paper,margin=2.84truecm]{geometry}
\usepackage{bbm}
\usepackage{soul}
\DeclareMathOperator{\diam}{diam}
\DeclareMathOperator{\dive}{div}

\def\ds{\displaystyle}
\def\eps{{\varepsilon}}

\def\O{\Omega}
\def\N{\mathbb{N}}
\def\R{\mathbb{R}}
\def\Z{\mathbb{Z}}
\def\A{\mathcal{A}}
\def\F{\mathcal{F}}

\def\la{\lambda}
\def\eps{\varepsilon}
\def\pa{\partial}

\newcommand{\ccap}{\mathrm{cap}}

\newcommand{\be}{\begin{equation}}
\newcommand{\ee}{\end{equation}}
\newcommand{\bib}[4]{\bibitem{#1}{\sc#2: }{\it#3. }{#4.}}

\numberwithin{equation}{section}
\theoremstyle{plain}

\newtheorem{theo}{Theorem}[section]
\newtheorem{lemm}[theo]{Lemma}

\newtheorem{prop}[theo]{Proposition}

\newtheorem{conj}[theo]{Conjecture}

\theoremstyle{definition}
\newtheorem{rem}[theo]{Remark}

\title[On a class of Cheeger inequalities]{On a class of Cheeger inequalities}

\author[L. Briani]{Luca Briani}

\author[G. Buttazzo]{Giuseppe Buttazzo}

\author[F. Prinari]{Francesca Prinari}

\date{}

\begin{document}

\begin{abstract}
We study a general version of the Cheeger inequality by considering the shape functional $\F_{p,q}(\O)=\lambda_p^{1/p}(\O)/\lambda_q^{1/q}(\O)$. The infimum and the supremum of $\F_{p,q}$ are studied in the class of {\it all} domains $\O$ of $\R^d$ and in the subclass of {\it convex} domains. In the latter case the issue concerning the existence of an optimal domain for $\F_{p,q}$ is discussed.
\end{abstract}

\maketitle



\section{Introduction}

The starting point of this research is the celebrated {\it Cheeger inequality}:
\be\label{cheeger}
\frac{\lambda(\O)}{h^2(\O)}\ge\frac14,
\ee
here $\lambda(\O)$ denotes the first eigenvalue of the Laplace operator $-\Delta$ on the open set $\O$, with Dirichlet boundary conditions, and $h(\O)$ denotes the {\it Cheeger constant}
\be\label{cheegerdef}
h(\O)=\inf\bigg\{\frac{P(E)}{|E|}\ :\ E\Subset\O\bigg\},
\ee
where the symbol $E\Subset\O$ indicates that the closure of $E$ is contained in $\O$. Here $P(E)$ denotes the {\it perimeter} of $E$ in the De Giorgi sense, and $|E|$ the Lebesgue measure of $E$. Equivalently $h(\O)$ can be defined through the expression
$$
h(\O)=\inf\bigg\{\frac{\int_\O|\nabla u|\,dx}{\int_\O|u|\,dx}\ :\ u\in C^1_c(\O)\bigg\}.
$$
With some additional regularity assumption on $\O$, in \eqref{cheegerdef} the infimum can be equivalently evaluated on the whole class of subsets $E\subset\O$. For instance, it is enough to require that $\O$ coincides with its essential interior; we refer the reader to \cite{leo15} and \cite{Pa} for a survey on the Cheeger constant.
We recall that if $\O$ is a ball of radius $r$ in $\R^d$ we have $h(\O)=d/r$.

In this paper we consider, for every $1<p<+\infty$, the $p$-Laplace operator
$$-\Delta_p u=-\dive\big(|\nabla u|^{p-2}\nabla u\big)$$
and the corresponding principal eigenvalue
\be\label{def.la}
\lambda_p(\O)=\inf\bigg\{\frac{\int_\O|\nabla u|^p\,dx}{\int_\O|u|^p\,dx}\ :\ u\in C^1_c(\O)\bigg\}.
\ee
The following properties are well-known:
\begin{itemize}
\item any minimizer of \eqref{def.la} solves, in the weak sense, the Dirichlet problem:
$$
\begin{cases}
-\Delta_p u=\lambda |u|^{p-2}u&\hbox{in }\O,\\
u\in W^{1,p}_0(\O);
\end{cases}
$$
\item$\la_p(\cdot)$ is {\it decreasing} with respect to the set inclusion, that is
\be\label{decla}
\la_p(\O)\le \la_p(\O'),\qquad\hbox{if }\O'\subset\O;
\ee
\item the {\it scaling property}
\be\label{scala}
\la_p(t\O)=t^{-p}\la_p(\O),\qquad\hbox{for all }t>0;
\ee
\item the {\it asymptotics}

\be\label{asyla}
\lim_{p\to+\infty}\la^{1/p}_p(\O)=\rho^{-1}(\O), \qquad \lim_{p\to 1^+}\la_p(\O)=h(\O),
\ee
where $\rho(\O)$ denotes the so-called {\it inradius} of $\O$, corresponding to the maximal radius of a ball contained in $\O$ (see \cite{JLM} and \cite{KaFr03}). Equivalently, $\rho(\O)$ can be defined as
$$\rho(\O):=\|d_\O\|_{L^\infty(\O)},$$
where $d_\O$ is the distance function from $\partial\O$
$$d_\O(x):=\inf\big\{|x-y|\ :\ y\in \partial\O \big\}.$$
\end{itemize}
Taking into account \eqref{asyla} we define 
\be\label{laextrem}
\la^{1/p}_p(\O)=\begin{cases}
h(\O) &\hbox{if }p=1;\\
\rho(\O)^{-1}&\hbox{if }p=+\infty.\\
\end{cases}
\ee
Inequality \eqref{cheeger} can be then seen as a particular case of the more general inequality
\be\label{cheegerp}
\frac{\lambda_p^{1/p}(\O)}{\lambda_q^{1/q}(\O)}\ge\frac{q}{p}\qquad\hbox{for every }1\le q\le p\le+\infty
\ee
that can be also rephrased as a monotonicity property:
$$\text{the map $p\mapsto p\lambda_p^{1/p}(\O)$ is monotonically increasing.}$$
Although this result is already known for $1<q\le p<+\infty$ (see \cite{Lind93}), for the sake of completeness we recall its proof in Proposition \ref{lindqvist}.

Our goal is to study from the shape optimization point of view the functionals
$$\F_{p,q}(\O)=\frac{\lambda_p^{1/p}(\O)}{\lambda_q^{1/q}(\O)}.$$
From the properties listed above $\F_{p,q}$ is \textit{scaling free}, that is
$$\F_{p,q}(t\O)=\F_{p,q}(\O)\qquad\text{for all }t>0.$$
We consider the minimization/maximization problem of $\F_{p,q}$ in the classes 
\[\begin{split}
&\A^d_{all}=\{\O\subset\R^d\ : \ \O\hbox{ open, }0<|\O|<+\infty\},\\
&\A^d_{convex}=\{\O\in\A^d_{all}\ :\ \O\hbox{ convex}\}.
\end{split}\]
For the sake of brevity we denote by $m_d(p,q),M_d(p,q)$ the quantities
$$m_d(p,q)=\inf_{\O\in\A^d_{all}}\F_{p,q}(\O),\qquad M_d(p,q)=\sup_{\O\in\A^d_{all}}\F_{p,q}(\O).$$
Similarly, for the convex case, we use the notation
$$\overline{m}_d(p,q)=\inf_{\O\in\A^d_{convex}}\F_{p,q}(\O),\qquad\overline{M}_d(p,q)=\sup_{\O\in\A^d_{convex}}\F_{p,q}(\O).$$

The study of the functionals $\F_{p,q}$ has been proposed in \cite{Pa15}, where the author focused on the case $p=2$, $q=1$. Recently some developments have been made in \cite{Fto}, again in the case $p=2$, $q=1$. 

The paper is organized as follows. In Section \ref{sall} we discuss the optimization problem in the class $\A^d_{all}$. In particular we prove that \eqref{cheegerp} becomes sharp when $d\to+\infty$ (Theorem \ref{bigdim}), and we characterize the behavior of $M_d(p,q)$ in varying $p,q$, showing that it remains finite if and only if $q>d$, (Theorem \ref{theo.supall}). The optimization problems in the class $\A^d_{convex}$ are discussed in Section \ref{sconvex}. After recalling some known estimates we prove that $\overline{M}_d(p,q)$ is always finite (Proposition \ref{convexbound}) and that, in some cases, the minimization problems for $\F_{p,q}$ among planar convex open sets, admits a solution, (Theorem \ref{main}). In Section \ref{sfurther}, we collect some open problems that in our opinion can be interesting for future researches. At last, we conclude the paper with a small appendix, where we give self contained proofs of some known facts in shape optimization, which are useful for our purpose.


\section{Optimization in $\A^d_{all}$}\label{sall}

As it often happens in shape optimization, the one-dimensional case is simpler. Indeed in this case the functional $\F_{p,q}$ turns out to be constant. Hereinafter we denote by $\pi_p$ the Poincar\'e-Sobolev constant:
\be\label{pi} \pi_p=\inf\left\{\frac{\|\phi'\|_{L^p(0,1)}}{\|\phi\|_{L^p(0,1)}}\ :\ \phi\in C^1_c(0,1),\ \phi(1)=\phi(0)=0\right\}=\la^{1/p}_p(0,1).
\ee
Explicit computations, see for instance \cite{Kaji}, show that
$$\pi_p=2\pi\frac{(p-1)^{1/p}}{p\sin(\pi/p)}.$$
In particular one has $\pi_2=\pi$, $\pi_1=\pi_\infty=2$, and $\pi_p=\pi_{p'}$ for every $p$, where $p'$ is the conjugate exponent of $p$.

\begin{prop}
Let $1\le q\le p\le+\infty$. Then, for every $\O\in\A^1_{all}$ we have
$$F_{p,q}(\O)=\frac{\pi_p}{\pi_q}.$$
\end{prop}

\begin{proof}
It is enough to notice that if $\O\in\A^1_{all}$ is the disjoint union of a family of open intervals $(\O_i)_{i\in I}$, then, for every $1\le p\le +\infty$, we have
\be\label{ladisj}
\lambda_p^{1/p}(\O)=\inf_{i\in I}\lambda_p^{1/p}(\O_i),
\ee
 Indeed, when $p=+\infty$ \eqref{ladisj} is straightforward by \eqref{laextrem}, while, when $1\le p<+\infty$, we notice that for every $u\in C^\infty_c(\O)$ it holds
$$\int_\O|\nabla u|^pdx\ge\sum_{i\in I}\int_{\O_i}|\nabla u|^pdx
\ge\sum_{i\in I}\la_p(\O_i)\int_{\O_i}|u|^pdx\ge\inf_{i\in I}\la_p(\O_i)\sum_{i\in I}\int_{\O_i}|u|^pdx,$$
which implies
$$\lambda_p(\O)\ge\inf_{i\in I}\lambda_p(\O_i).$$
By \eqref{decla}, the latter inequality easily leads to \eqref{ladisj}.
Taking into account that, by \eqref{scala} and \eqref{pi}, we have
$$\inf_{i\in I}\lambda_p^{1/p}(\O_i)=\inf_{i\in I}|\O_i|^{-1}\pi_p,$$
we achieve the thesis.
\end{proof}

From now on we always assume $d\ge2$. Next proposition provides a lower bound to $m_d(p,q)$ and generalizes inequality \eqref{cheeger}.

\begin{prop}\label{lindqvist}
Let $\O\in\A^d_{all}$. Then, the function $p\mapsto p\la_p^{1/p}(\O)$ is nondecreasing in $[1,+\infty]$. In particular we have
\be\label{infall}
m_d(p,q)\ge q/p.
\ee
\end{prop}

\begin{proof}
By \eqref{asyla} it is enough to consider the case $1<q<p<\infty$. Let $u\in C^\infty_c(\O)$ and let $v=u^{p/q}$. Then, by H\"older inequality, we get
\[\begin{split}
\la_q(\O)&\le\frac{\int_\O|\nabla v|^qdx}{\int_\O|v|^qdx}=\left(\frac pq\right)^q\frac{\int_\O|\nabla u|^q|u|^{p-q}dx}{\int_\O|u|^pdx}\\
&\le\left(\frac pq\right)^q\frac{\left(\int_\O|\nabla u|^pdx\right)^{q/p}}{\int_\O|u|^pdx}\left(\int_\O|u|^pdx\right)^{1-q/p}=\left(\frac pq\right)^q\left(\frac{\int_\O|\nabla u|^pdx}{\int_\O|u|^pdx}\right)^{q/p}.
\end{split}\]
Since $u$ is arbitrary we obtain
$$q\la^{1/q}_{q}(\O)\le p\la^{1/p}_{p}(\O)$$
as required.
\end{proof}

In general, we do not expect the bound given in \eqref{infall} to be sharp. For instance, as $p\to+\infty$, the right-hand side in \eqref{infall} tends to zero, while it is easy to prove that the minimum of $\F_{\infty,q}$ is strictly positive and attained at any ball. Indeed, since any $\O$ contains a ball of radius $\rho(\O)$, by \eqref{decla}, \eqref{scala} and \eqref{asyla}, we have 
\be\label{laradmon}
\la^{1/q}_q(\O)\le\rho^{-1}(\O)\la_q^{1/q}(B_1)\quad \hbox{for every } 1\le q \le +\infty,
\ee
which clearly implies
$$m_d(\infty,q)=\F_{\infty,q}(B_1), \qquad\hbox{for every }1\le q< +\infty;$$
here we denote by $B^d_r$ the ball in $\R^d$ of radius $r$ centered at the origin, and we omit the dependence on $d$ when there is no ambiguity.

Recently, by exploiting the fact that $\la_2(B_1^d)=j_{d/2-1,1}$, where $j_{d/2-1,1}$ denotes the first root of the $d$-th Bessel function of first kind, Ftouhi (see \cite{Fto}) has noticed that
\be\label{fteq}
\lim_{d\to+\infty}m_{d}(2,1)=1/2.
\ee 
Our next goal is to generalize the limit \eqref{fteq} to every $p,q$. With this aim we introduce the quantity
$$B(s,t)=\int_0^1 \tau^{s-1}(1-\tau)^{t-1}d\tau$$
and recall that, in terms of the Euler's function $\Gamma$, we have
\be\label{gammabeta}
B(s,t)=\frac{\Gamma(s)\Gamma(t)}{\Gamma(s+t)}.
\ee

\begin{lemm}
Let $\O\in\A^d_{all}$ and $s\ge1$. Then,
\be\label{stimaEuler}
\la_p(B_1)\le s^p\left(\frac{\Gamma(sp+d+1)\Gamma(sp-p+1)}{\Gamma(sp+1)\Gamma(sp+d-p+1)}\right).
\ee
\end{lemm}

\begin{proof}
Let $s\ge 1$ and $\phi(x)=(1-|x|)^s$. Clearly $\phi\in W^{1,p}_0(B_1)$ and
$$\int_{B_1}|\phi(x)|^p dx=d\omega_d\int_0^1(1-t)^{sp}t^{d-1}dt=d\omega_d B(d,sp+1).$$
Similarly we have
$$\int_{B_1}|\nabla\phi(x)|^p dx=d\omega_d s^p\int_0^1(1-t)^{(s-1)p}t^{d-1}dt=d\omega_ds^pB(d,sp-p+1).$$
Now, using $\phi$ as a test function in \eqref{def.la}, we obtain
$$\la_p(B_1)\le s^p\left(\frac{B(d,sp-p+1)}{B(d,sp+1)}\right).$$
Finally, \eqref{stimaEuler} follows from \eqref{gammabeta}.
\end{proof}

\begin{lemm}\label{cilindri}
Let $1\le p<+\infty$, $L>0$, and $\omega\in \A_{all}^{d-1}$. Denote by $\O_L=\omega\times(-L/2,L/2)$. Then
\be\label{cilindristime}
\begin{cases}
\la_p(\omega)+\pi_p^p/L^p \le \la_p(\O_L)\le\big(\la^{2/p}_p(\omega)+\pi_p^2/L^2\big)^{p/2}&\hbox{if }p\ge2,\\
\big(\la^{2/p}_p(\omega)+\pi_p^2/L^2\big)^{p/2}\le \la_p(\O_L)\le\la_p(\omega)+\pi_p^p/L^p&\hbox{if }p\le2.
\end{cases}
\ee
In particular
\be\label{limitecilindri}
\lim_{L\to+\infty}\la^{1/p}_p(\O_L)=\la^{1/p}_p(\omega).
\ee
\end{lemm}

\begin{proof}
We denote by $(x,y)$ the points in $\R^{d-1}\times\R$. Let $u\in C^{\infty}_c(\O_L)$, then for every $(x,y)\in\O$ we have
$$u(\cdot,y)\in C^{\infty}_c(\omega),\qquad u(x,\cdot)\in C^{\infty}_c(-L/2,L/2).$$
If $p\ge2$, using the super-additivity of the function $t\to t^{p/2}$ and Fubini theorem, we have
\[\begin{split}
\int_{\O_L}|\nabla u|^pdxdy&=\int_{-L/2}^{L/2}\int_\omega\left(|\nabla_x u|^2+|\pa_yu|^2\right)^{p/2}dxdy\\
&\ge\left(\la_p(\omega)+\frac{\pi_p^p}{L^p}\right)\int_{\O_L}|u|^pdxdy.
\end{split}\]
Similarly, if $p\le2$, using Fubini theorem together with the reverse Minkowski inequality
$$\|f+g\|_{L^{p/2}(\O_L)}\ge \|g\|_{L^{p/2}(\O_L)}+\|f\|_{L^{p/2}(\O_L)},$$
 we obtain
\[\begin{split}
\int_{\O_L}|\nabla u|^pdxdy&=\int_{-L/2}^{L/2}\int_\omega\left(|\nabla_x u|^2+|\pa_yu|^2\right)^{p/2}dxdy\\
&\ge\left\{\left(\int_{\O_L}|\nabla_x u|^pdxdy\right)^{2/p}+\left(\int_{\O_L}|\pa_y u|^pdxdy\right)^{2/p}\right\}^{p/2}\\
&\ge\left(\la_p(\omega)^{2/p}+\frac{\pi_p^2}{L^2}\right)^{p/2}\int_{\O_L}|u|^pdxdy.
\end{split}\]
In both cases, the arbitrariness of $u$ proves the left hand side inequalities in \eqref{cilindristime}.

The upper estimates in \eqref{cilindristime} can be proved with the same argument, once chosen a suitable test function. More precisely, we take $u(x)$ and $v_L(y)$ optimal functions respecively for $\la_p(\omega)$ and $\la_p((-L/2,L/2))$, with unitary $L^p$ norm, that is (taking also \eqref{scala} into account) we require:
\be\label{optim}
\|\nabla u\|^p_{L^p(\omega)}=\la_p(\omega),\qquad\|v_L'\|^p_{L^p(-L/2,L/2)}=\pi_p^p/L^p.
\ee
Now, the product function $\phi(x,y)=u(x)v_L(y)$ is admissible in the computation of $\la_q(\O_L)$ and gives
$$\la_p(\O_L)\le\int_{\O_L}|\nabla\phi(x,y)|^pdxdy=\int_{\O_L}\big(|\nabla_x u(x)v(y)|^2+|u(x)v'(y)|^2\big)^{p/2}dxdy.$$
If $p\le2$, by the sub-additivity of the function $t\to t^{p/2}$, \eqref{optim} and Fubini theorem, we get 
$$\la_p(\O_L)\le\la_p(\omega)+\frac{\pi^p}{L^p}.$$
Similarly if $p\ge2$, by \eqref{optim}, Fubini theorem and Minkowski inequality we have that
$$\la_p(\O_L)\le\left(\la^{2/p}_p(\omega)+\frac{\pi_p^2}{L^2}\right)^{p/2},$$
which concludes the proof.
\end{proof}

\begin{rem}
The limit \eqref{limitecilindri} is clearly true also when $p=+\infty$, since in this case $\rho(\O_L)=\rho(\omega)$, as soon as $L$ is large enough.
\end{rem}

We may now prove the general form of limit \eqref{fteq}.

\begin{theo} \label{bigdim}
Let $1\le q< p\le+\infty$. Then the sequence $d\mapsto m_d(p,q)$ is nonincreasing and
\be\label{limitemd}
\lim_{d\to+\infty}m_d(p,q)=\inf_{d\geq 1 } m_d(p,q)= q/p.
\ee 
In particular,
$$\frac{q}{p}\le m_d(p,q)\le m_1(p,q)=\frac{\pi_p}{\pi_q}.$$
\end{theo}

\begin{proof}
The monotonicity of the sequence follows at once by \eqref{limitecilindri}, hence the limit above exists as well. In order to prove \eqref{limitemd}, first we suppose $q=1$. By applying \eqref{stimaEuler} with $s=\sqrt{d}$, we get
$$\F_{p,1}(B_1^d)=\frac{\la^{1/p}_p(B_1^d)}{d}\le\frac{1}{\sqrt{d}}\cdot\left(\frac{\Gamma(\sqrt{d}p+d+1)\Gamma(\sqrt{d}p-p+1)}{\Gamma(\sqrt{d}p+1)\Gamma(\sqrt{d}p+d-p+1)}\right)^{1/p}.$$
Moreover, using the fact that $\Gamma(s+t)\approx\Gamma(s)s^t $ as $s\to\infty$, we obtain that, as $d\to\infty$,
$$
\frac{1}{\sqrt{d}}\left(\frac{\Gamma(\sqrt{d}p+d+1)\Gamma(\sqrt{d}p-p+1)}{\Gamma(\sqrt{d}p+1)\Gamma(\sqrt{d}p+d-p+1)}\right)^{1/p}\approx\frac{1}{\sqrt{d}}\left(1+\frac{\sqrt{d}}{p}\right).
$$
Hence, by applying also \eqref{infall}, we obtain
$$1/p\le\lim_{d\to\infty} m_d(p,1)\le\limsup_{d\to\infty}\F_{p,1}(B_1^d)\le1/p.$$
To achieve the general case we notice that, for all $\O\subset\R^d$, we have
$$q/p\le m_d(p,q)\le\F_{p,q}(\O)=\F_{p,1}(\O)\F^{-1}_{q,1}(\O)\le q\F_{p,1}(\O),$$
where the last inequality follows again by \eqref{infall}. Then
$$q/p\le\lim_{d\to\infty}m_d(p,q)\le q\lim_{d\to\infty}m_d(p,1)=q/p$$
as required. Finally, the last statement is an easy consequence of \eqref{limitemd}.
\end{proof}

We now turn our attention to the quantity $M_d(p,q)$ and we notice that limit \eqref{limitecilindri} also implies that the sequence $d\mapsto M_d(p,q)$ is nondecreasing and hence
$$\frac{\pi_p}{\pi_q}=M_1(p,q)\le M_d(p,q).$$ 
Our next result deals with the upper bound for $M_d(p,q)$. We recall that the (relative) $p$-capacity of a set $E\subset\O$ is defined as
$$
\ccap_p(E;\O)=\inf\left\{\int_\O|\nabla u|^pdx\ :\ u\in W^{1,p}_0(\O),\ u\ge 1\ \hbox{a.e. in a neighborhood of }E\right\}.
$$
A set $E\subset\R^d$ is said to be of zero $p$-capacity if
$$\ccap_p(E\cap\O;\O)=0\qquad\hbox{for all }\O\in\A^d_{all};$$
in this case we simply write $\ccap_p(E)=0$. For a comprehensive introduction to $p$-capacity we refer the reader to \cite{Hein} and \cite{Maz}. Here we only point out that, given $1< p<+\infty$ and $E$ a relatively closed subset of $\O$, then
$$\ccap_p(E)=0\ \Longrightarrow\la_p(\O\setminus E)=\la_p(\O).$$ 
Moreover, using the fact that when $p>d$ even a single point has nonzero $p$-capacity, in \cite{Poli} it is shown the following.

\begin{theo}\label{Poli}
Let $d\in\N$, $d\ge 1$ and $d<p<+\infty$. There exists a positive constant $C_{p,d}$, depending on $p$ and $d$, such that for every bounded open set $\O\subset\R^d$ it holds
$$\la^{1/p}_p(\O)\ge C_{d,p}\rho^{-1}(\O).$$
\end{theo}

\begin{rem}
Theorem \ref{Poli} can be extended to the whole class $\A^d_{all}$ by means of a simple approximation argument. Indeed, it is sufficient to note that, if $\O\in\A^d_{all}$ is unbounded, and we set $\O_n:=\O\cap B_n$, it holds
$$\lim_{n\to+\infty}\rho(\O_n)=\rho(\O),\qquad\la_p(\O)=\lim_{n\to+\infty}\la_p(\O_n),$$
and, by Theorem \eqref{Poli},
$$\la_p(\O_n)\ge C_{d,p}\rho^{-p}(\O_n)\qquad\hbox{ for every }n\in\N.$$
Passing to the limit as $n\to+\infty$ in the inequality above gives the conclusion.
\end{rem}

\begin{theo}\label{theo.supall}
Let $1\le q<p\le\infty$. Then
$$
\begin{cases}
\ds M_d(p,q)<\frac{\la_p^{1/p}(B_1)}{C_{d,q}}&\hbox{if }d< q,\\
M_d(p,q)=+\infty&\hbox{otherwise,}
\end{cases}
$$
where $C_{d,q}$ is the constant given by Theorem \ref{Poli}.
\end{theo}

\begin{proof}
The case when $d<q$ follows by combining Theorem \ref{Poli} (applied to $\la_q$) and inequality \eqref{laradmon} (applied to $\la_p$). 

The case $1\le q\le d<p\le \infty$ is a consequence of the fact that if $1<q\le d<p$ then a single point has zero $q$-capacity. More precisely, let $(x_n)$ be a dense sequence in a ball $B\subset\R^d$ and define 
$$\O_n:=B\setminus\bigcup_{i=1}^{n}\{x_i\}.$$
Since $\ccap_q(\bigcup_{i=1}^{n}\{x_i\})=0$, we have $\la_q(\O_n)=\la_q(B)$ for every $n\in\N$. Taking into account \eqref{asyla}, we have also $h(\O_n)=h(B)$ for every $n\in\N$. Moreover, since $\rho(\Omega_n)\to0$, by Theorem \ref{Poli} or \eqref{laextrem}, we have that $\la^{1/p}_p(\O_n)\to+\infty$. Therefore $\F_{p,q}(\O_n)\to+\infty$ for every $1\le q<d<p$.

The case when $1\le q<p\le d$ is more delicate. Our argument is inspired by the example exhibited in the Appendix A of \cite{BS}. Given $1<p\leq d$ we construct a sequence of open bounded sets $\O_n\subset\R^d$ such that for every $q<p$
$$\lim_{n\to\infty}\F_{p,q}(\O_n)=+\infty.$$
Let $Q= (-1/2,1/2)^d$. Being $1<p\leq d$, it is well known that there exists a compact set $E_p\subset[0,1]^d$ such that $\ccap_p(E_p)>0$ and $\ccap_q(E_p)=0$ when $1<q<p$ (see Lemma 7.1 in \cite{Lind93}; for instance, $E_p$ can be constructed as a Cantor set). By translating and rescaling the compact set $E_p$, we can assume that $E_p\subset[-1/4,1/4]^d$. Then we consider the open sets
$$\O_n=(-(n+1/2),n+1/2)^d\setminus\bigcup_{z\in\Z_n^d}(E_p+z),$$
where $\Z^d_n=\Z^d\cap[-n,n]^d$ and
$$E=\bigcup_{n\in \N}\O_n=\R^d\setminus \bigcup_{z\in\Z^d}(E_p+z).$$
Being $\ccap_p(E_p)>0$, by Theorem 10.1.2 in \cite{Maz}, we have that
\be\label{Maz}
\min\bigg\{\frac{\int_Q|\nabla u|^p\,dx}{\int_Q|u|^p\,dx}\ :\ u\in W^{1,p}(Q),\ u=0\hbox{ on }E_p\bigg\}= C(d,p,E_p)>0.
\ee
Now, since any function $u\in C_c^{\infty}(E)$ when restricted to $Q+z$, with $z\in \Z^d$, vanishes on a translated copy of $E_p$, \eqref{Maz} readily implies 
$$\la_p(E)\ge C(d,p,E_p).$$
Then, by monotonicity we have
\be\label{p}
\la_ p(\O_n)\ge\la_p(E)\ge C(d,p,E_p)>0,
\ee
for every $n\in\N$. Moreover, for every $q>1$, being $\ccap_q(E_p)=0$, we have that
$$\la_q(\O_n)=\la_q((-(n+1/2),n+1/2)^d)$$
and hence
$$h(\O_n)= h((-(n+1/2),n+1/2)^d)$$
as well. This gives
\be\label{q}
\la^{1/q}_q(\O_n)=(2n+1)^{-1} \la^{1/q}_q(Q)\to0\qquad\hbox{as }n\to+\infty
\ee
for every $q\ge1$. By combining \eqref{p} and \eqref{q} the thesis is easily achieved.
\end{proof}

\begin{rem}\label{cioranescuproof}
The case $1\le q< p\le d$ in the previous theorem can be also proved by constructing a sequence $\O_n$ satisfying:
\be\label{cioranescustrat}
\la_q(\O_n)\to \la_q(B_1), \qquad \la_p(\O_n)\to+\infty.
\ee
To do this, one can consider the sequence $\O_n$ obtained by removing from the unit ball a periodic array of spherical holes of size $r_n$, where
$$\begin{cases}
n^{d/(p-d)}\ll r_n\ll n^{d/(q-d)}&\text{if }p<d;\\
e^{-n^{d/(d-1)}}\ll r_n\ll n^{d/(q-d)}&\text{if }p=d.
\end{cases}$$
Then classical results of shape optimization theory can be used to get \eqref{cioranescustrat} (see \cite{ciomur} and references therein). We devote Appendix \ref{sapp} to give a self-contained proof.
\end{rem}


\section{Optimization in $\A^d_{convex}$}\label{sconvex}

In this section we consider the optimization problems in the class $\A^d_{convex}$. We remark that also in this case Lemma \ref{cilindri} provides the monotonicity properties:
$$d\mapsto \overline{m}_d(p,q)\ \hbox{is nonincreasing} \hbox{ and }d\mapsto \overline{M}_d(p,q)\ \hbox{is nondecreasing}.$$
To carry on our analysis we use two fundamental inequalities which hold for every $\O\in\A^d_{convex}$:
\begin{itemize}
\item the \textit{Hersch-Protter inequality}:
\be\label{HP1}
\rho(\O)\lambda^{1/p}_p(\O)> \frac {\pi_p}{2};
\ee
\item the \textit{Buser inequality}:
\be\label{Buser}
\frac{\lambda_p^{1/p}(\O)}{h(\O)}< \frac{\pi_p}{2}.
\ee
\end{itemize}
Inequality \eqref{HP1} was first proved in \cite{Her} and \cite{Prot} when $p=2$, and then extended to general case in \cite{bra18}, while inequality \eqref{Buser} is proved in \cite{Pa15} in the planar linear case, and in \cite{bra20} in the general one. Both inequalities are sharp, as one can verify by taking a sequence of thin slab domains $\O_n:=[0,1]\times[0,1/n]$, see for instance \cite{bra18} and \cite{bra20}. As a consequence one has that
$$\overline{M}_d(p,1)=\frac{\pi_p}{2}, \qquad \overline{M}_{d}(\infty,q)=\frac{2}{\pi_q},$$
so that the following conjecture formulated by Parini in \cite{Pa15}, is satisfied in the particular cases $p=+\infty$ or $q=1$.

\begin{conj}\label{conj}
Let $1\le q<p\le +\infty$. Then we have
$$\overline{M}_{d}(p,q)=\frac{\pi_p}{\pi_q},$$
and no maximizer set exists. 
\end{conj}

Although we are not able prove the conjecture we show the following estimates.

\begin{prop}\label{convexbound}
Let $1\le q<p\le+\infty$. Then, for all $\O\in\A^d_{convex}$ we have 
$$\max\Big\{\frac{q}{p},\frac{\pi_p}{d\pi_q}\Big\}\le\overline{m}_d(p,q)\le\overline{M}_d(p,q)\le\pi_p\min\Big\{\frac{q}{2},\frac{d}{\pi_q}\Big\}.$$
\end{prop}

\begin{proof}
We first notice that, being $h(B_1)=d$, inequality \eqref{laradmon} with $p=1$ provides
$$h(\O)\rho(\O)\le d.$$
Hence, by using \eqref{HP1} (with $p$) and \eqref{Buser} (with $q$), we obtain
\be\label{upperBH}
\F_{p,q}(\O)=\frac{\la_p^{1/p}(\O)}{\la_q^{1/q}(\O)}\leq \frac{\pi_p}{\pi_q} h(\O)\rho(\O)\le \frac{d\pi_p}{\pi_q}.
\ee
By interchanging the role of $p$ and $q$, we get
\be\label{lowerBH}
\F_{p,q}(\O)=\frac{\la_p^{1/p}(\O)}{\la_q^{1/q}(\O)}\ge\frac{\pi_p}{\pi_q }\frac{1}{h(\O)\rho(\O)}\ge\frac{\pi_p}{d\pi_q}.
\ee
Inequalites \eqref{lowerBH} and \eqref{cheegerp} prove that
$$\max\Big\{\frac{q}{p},\frac{\pi_p}{d\pi_q}\Big\}\le\overline{m}_d(p,q),$$
while, using \eqref{Buser} and \eqref{cheegerp}, we have
$$\la_p^{1/p}(\O)\leq \frac{\pi_p}{2}h(\O)\le q\frac{\pi_p}{2}\la_q^{1/q}(\O),$$
which, together with \eqref{upperBH}, implies
$$\overline{M}_d(p,q)\le\pi_p\min\Big\{\frac{q}{2},\frac{d}{\pi_q}\Big\}$$
as required.
\end{proof}

In \cite{Pa15} it is proved that the functional $\F_{2,1}$ admits a minimizing set in the class of bounded convex planar domains. Recently in \cite{Fto}, the author discussed the existence of minimizers for $\F_{2,1}$ in $\A^d_{convex}$ for $d\ge3$, which, up to our knowledge, remains open. In Theorem \ref{main} below we show the existence of a minimizer for $\F_{p,q}$ in the class $\A_{convex}^2$ when $q\le2\le p$. Before proving the theorem we need some preliminary results, that we state in the general case of dimension $d$.

\begin{lemm}\label{BrascoDePhillips}
Let $1\le p\le+\infty$ and $\O\in\A^d_{convex}$. Let $a=(0,\dots,0)$, $b=(0,\dots,\diam(\O))$, and suppose $a,b\in\pa\O$. Then there exists $0<t<\diam(\O)$ such that
$$\la^{1/p}_p(\O)\ge\la^{1/p}_p(\O\cap\{x_d=t\}),$$
where in the right-hand side $\la_p(\O\cap\{x_d=t\})$ is intended in the $d-1$ dimensional sense.
\end{lemm}

\begin{proof}
The case when $p=+\infty$ is trivial and hence we can suppose $1\le p<\infty$. We notice that there exists $t\in (0,\diam(\O))$ such that
$$\la_p(\O\cap\{x_d=t\})=\inf_{\tau\in (0,\diam(\O))}\la_p(\O\cap\{x_d=\tau\}).$$
Indeed the map
$$
\tau\mapsto\{\O\cap\{x_d=\tau\}\}\subset\R^d,
$$
is continuous with respect to the Hausdorff distance, and thus, thanks to the well-known continuity properties for $\la_p$ with respect to Hausdorff metrics on the class of bounded convex sets (see \cite{bubu05} and \cite{He}, for details about this fact), the map
$$
\tau\mapsto\la_p(\O\cap\{x_d=\tau\}),
$$
is continuous as well.
Moreover both $\O\cap\{x_d=0\}$ and $\O\cap\{x_d=\diam(\O)\}$ are empty, so that
$$\lim_{t\to 0^+}\la_p(\O\cap\{x_d=\tau\})=\lim_{\tau\to \diam(\O)^-}\la_p(\O\cap\{x_d=\tau\})=+\infty.$$
Now, let $\eps>0$ and $\phi\in C^1_c(\O)$ be such that $\|\phi\|_p=1$ and $\eps+\la_p(\O) \ge \|\nabla\phi\|_p^p$. Then
\[\begin{split}
\eps+\la_p(\O)&\ge\int_\O|\nabla\phi|^pdx=\int_0^{\diam(\O)}\left(\int_{\O\cap\{x_d=\tau\}}|\nabla\phi|^pdx'\right)d\tau\\
&\ge\int_0^{\diam(\O)}\left(\int_{\O\cap \{x_d=\tau\}}|\nabla_{x'}\phi|^pdx'\right)d\tau\\
&\ge \int_0^{\diam(\O)}\left(\la_p(\O\cap \{x_d=\tau\})\int_{\O\cap \{x_d=\tau\}}|\phi|^pdx'\right)d\tau\\
&\ge \la_p(\O\cap\{x_d=t\}),
\end{split}
\]
which, by the arbitrariness of $\eps$, implies the thesis.

\end{proof}

\begin{lemm}\label{sezioni}
Let $1\le p\le+\infty$ and $(\O_n)\subset\A^d_{convex}$ with $|\O_n|=1$ for every $n\in\N$. Suppose that
\begin{itemize}
\item $(0,\dots,0), (0,\dots,\diam(\O_n))\in\pa\O_n,$
\item $\inf_{n\in\N}\diam (\O_n)>0.$
\end{itemize}
Then, there exists $c>0$ such that $$\inf_{n\in\N}\inf_{\tau\in (0,\diam(\O_n))}\la^{1/p}_p(\O_n\cap\{x_d=\tau\})\ge c.$$
\end{lemm}

\begin{proof}
For any $n\in\N$ and any $t\in (0,\diam(\O_n))$, we denote 
$$\omega_{n}(t)=\O_n\cap\{x_d=t\}\in \A_{convex}^{d-1}.$$
By \eqref{HP1} we have
$$
\la^{1/	p}_p(\omega_n(t))> \frac{\pi_p}{2\rho(\omega_n(t))}.
$$
Being $\O_n$ convex, the cone set having basis $\omega_n(t)$ and height $t\ge \diam(\O_n)/2$ is contained in $\O_n$. Hence we have
$$
\frac{\rho^{d-1}(\omega_n(t))|B^{d-1}_1|t}{d}\le |\O_n|=1.
$$
In particular 
$$
\rho(\omega_n(t))\le \left(\frac{2d}{|B^{d-1}_1|\diam(\O_n)}\right)^{1/(d-1)}.
$$
Since $\inf_{n\in\N}\diam(\O_n)>0$, we obtain
$$\rho(\omega_n(t))\le\left(\frac{2d}{|B^{d-1}_1|\inf_n\diam(\O_n)}\right)^{1/(d-1)},$$
and the thesis easily follows.
\end{proof}

\begin{prop}\label{Ftouhi}
Let $1\le q<p\le+\infty$ and $(\O_n)\subset\A^d_{convex}$ with $|\O_n|=1$ for every $n\in\N$. If $\diam(\O_n)\to+\infty$, then
$$\overline{m}_{d-1}(p,q)
\le\liminf_{n\to+\infty}\F_{p,q}(\O_n)\le\limsup_{n\to+\infty}\F_{p,q}(\O_n)\le\overline{M}_{d-1}(p,q).$$
\end{prop}

\begin{proof}
Let $a_n=(0,\dots,0)$ and $b_n=(0,\dots, \diam(\O_n))$. Being the functional $\F_{p,q}$ rotations and translations invariant we can suppose, without loss of generality, that $a_n,b_n\in \pa\O_n$. By applying Lemma \ref{BrascoDePhillips} there exists $t_n\in (0,\diam(\O_n))$ such that
$$\la^{1/p}_p(\O_n)\ge \la^{1/p}_p(\omega_n),$$
where $\omega_n=\O_n\cap \{x_d=t_n\}$.
Moreover, we can also suppose $t_n\ge\diam(\O_n)/2$. Let $\alpha\in(0,1)$, and define $U^n_\alpha$ to be the cylinder with basis $\alpha\omega_n$ and height $(1-\alpha)t_n$. More precisely we consider
$$U^n_\alpha:=\{(x,y)\ :\ x\in\alpha\omega_n,\ y\in(\alpha t_n,t_n)\}.$$
Then, by the convexity of $\O_n$ we have $U_\alpha^n\subset\O_n$ so that
$$\la_q(U_\alpha^n)\ge\la_q(\O_n).$$
Since 
$$\la_q(U_\alpha^n)=\alpha^{-q}\la_q\left(\omega_n\times (0, \frac{(1-\alpha)}{\alpha}t_n)\right)$$
we obtain
\[\begin{split}
\F_{p,q}(\O_n)&\ge\frac{\la^{1/p}_p(\O_n)}{\la^{1/q}_q(U_\alpha^n)}\ge\frac{\la^{1/p}_p(\omega_n)}{\la_q^{1/q}(\omega_n)}\left(\frac{\la_q(\omega_n)}{\alpha^{-1}\la_q(\omega_n\times (0,1-\alpha)\alpha^{-1}t_n)}\right)^{1/q}\\
&\ge\alpha\overline{m}_{d-1}(p,q)\left(\frac{\la_q(\omega_n)}{\la_q(\omega_n\times(0,1-\alpha)\alpha^{-1}t_n)}\right)^{1/q}.
\end{split}\]
Now, suppose that $q\le 2$. By Lemma \ref{cilindri} we have
$$
\la_q(\omega_n\times (0,1-\alpha)\alpha^{-1}t_n)\le\la_q(\omega_n)+\left(\frac{\alpha\pi_q}{(1-\alpha)t_n}\right)^q.
$$
Since $\diam(\O_n)\to+\infty$, we can assume $\inf_n\diam(\O_n)>0$; by applying Lemma \ref{sezioni} we have that $\la^{1/q}_q(\omega_n) t_n\ge ct_n $ for some constant $c>0$, that implies
$$\lim_{n\to+\infty}\la^{1/q}_q(\omega_n) t_n=+\infty.$$
Then
$$\lim_{n\to+\infty}\frac{\la_q(\omega_n)}{\la_q(\omega_n)+\displaystyle\frac{\alpha^q\pi^q_q}{(1-\alpha)^qt_n^q}}=1.$$
This allows to conclude that
$$\liminf_{n\to+\infty}\F_{p,q}(\O_n)\ge\alpha\overline{m}_{d-1}(p,q)$$
and finally, letting $\alpha\to 1^-$, we conclude. The case when $q\ge2$ is similar. Indeed, \eqref{cilindristime} ensures that
$$
\la_q(\omega_n\times (0,1-\alpha)\alpha^{-1}t_n)\le\la_q(\omega_n)\left(1+\frac{\alpha^2\pi^2_q}{(1-\alpha)^2\la^{2/q}_q(\omega_n)t_n^2}\right)^{q/2},
$$
and again Lemma \ref{sezioni} applies.

Finally, if we choose $\omega_n$ to be such that $\la_q(\omega_n)\le \la_q(\O_n)$, and we use the fact that $\la^{1/p}_p(\O_n)\le \la_p^{1/p}(U^n_\alpha)$ we obtain:
$$\F_{p,q}(\O_n)\le \frac{\la^{1/p}_p(U^n_\alpha)}{\la^{1/q}_q(\Omega_n)}\le \frac{\overline{M}_{d-1}(p,q)}{\alpha}\left(\frac{\la_p(\omega_n\times (0, \alpha^{-1}(1-\alpha)t_n)}{\la_p(\omega_n)}\right)^{1/p}.$$
By the same argument as above, passing to the limit, as $n\to\infty$, we have
$$\limsup_{n\to\infty}\F_{p,q}(\O_n)\le\alpha^{-1}\overline{M}_{d-1}(p,q)$$
being $\omega_n$ an open convex set of $\R^{d-1}$. Finally, letting $\alpha\to1^-$, we conclude that 
$$\limsup_{n\to\infty}\F_{p,q}(\O_n)\le\overline{M}_{d-1}(p,q)$$
as required.
\end{proof}

\begin{theo}\label{existence}
Let $1\le q<p\le+\infty$. If $\overline{m}_{d}(p,q)<\overline{m}_{d-1}(p,q)$, then there exists $\O^d_\star\in \A_{convex}^d$ such that
$$\overline{m}_d(p,q)=\F_{p,q}(\O^d_\star).$$
\end{theo}

\begin{proof}
Let $(\O_n)$ be such that $\F_{p,q}(\O_n)\to \overline{m}_d(p,q)$ with $|\O_n|=1$ for every $n\in\N$. Then, by Proposition \ref{Ftouhi}, we have
$$\sup_n\diam(\O_n)<+\infty.$$
Hence, up to translations, the whole sequence $(\O_n)$ is contained in a compact set and we can extract a subsequence $(\O_{n_k})$ which converges in the Hausdorff distance to some $\O^d_{\star}$. Using the continuity properties for $\la_p$ with respect to Hausdorff metrics on the class of bounded convex sets, we have
$$\overline{m}_d(p,q)=\lim_{n\to\infty}\F_{p,q}(\O_n)=\F_{p,q}(\O^d_{\star})$$
as required.
\end{proof}

\begin{lemm}\label{stimequadrati}
Let $1\le p<\infty$ and $Q=(0,1)^d$ be the unitary cube of $\R^d$. Then
\[\begin{split}
d^{1/p}\pi_p\le\la^{1/p}_p(Q)< d^{1/2}\pi_p\qquad\text{for every }p>2;\\
d^{1/2}\pi_p< \la^{1/p}_p(Q)\le d^{1/p}\pi_p\qquad\text{for every }p<2.\\
\end{split}\]
\end{lemm}

\begin{proof}
By Lemma \ref{cilindri} (applied $d$ times and with $L=1$) we need only to prove the two strict inequalities. With this aim we define
$$\nu_p(Q)=\inf_{\phi\in C^{\infty}_c(Q)\setminus\{0\}}\frac{\int_{Q}\sum_{i=1}^d\left|\frac{\pa\phi}{\pa x_i}\right|^pdx}{\int_{Q}|\phi|^pdx}.$$
We notice that $\nu_p(Q)=d\pi_p^p$, with a minimizer given by
\be\label{disac}
\phi(x_1,\dots,x_d)=u(x_1)\cdots u(x_d),
\ee
where $u\in W^{1,p}(0,1)$ is a non negative function, optimal for \eqref{pi}, with unitary $L^p$ norm. 
Now, the case when $p>2$ follows by strict convexity of the map $t\to t^{p/2}$: indeed, being $d\ge 2$, integrating over $Q$ the inequality
$$|\nabla\phi(x)|^p<d^{p/2-1}\sum_i\left|\frac{\pa\phi(x)}{\pa x_i}\right|^p,$$
we obtain
$$
\la_p(Q)<d^{p/2-1}\nu_p(Q)=d^{p/2}\pi_p^p.
$$

Similarly, when $p<2$, we can consider $\tilde \phi$ to be the optimal positive function for $\la_p(Q)$, with unitary $L^p$ norm. Then, being $d\ge 2$, the strict concavity of the map $t\to t^{p/2}$ gives
$$
|\nabla \tilde \phi(x)|^p> d^{p/2-1}\sum_i\left|\frac{\pa\tilde \phi(x)}{\pa x_i}\right|^{p},
$$
which, integrated over $Q$, implies that
$$
\la_p(Q)>d^{p/2-1}\nu_p(Q)=d^{p/2}\pi_p^p.
$$
This concludes the proof of the lemma.
\end{proof}

We are now in a position to prove the following existence result in the case $d=2$.

\begin{theo}\label{main}
Let $1\le q<p\le +\infty$. Suppose $q\le2\le p$, then there exists $\O^\star\in\A^2_{convex}$ such that
$$\F_{p,q}(\O^\star)=\min\{\F_{p,q}(\O)\ :\ \O\in\A^2_{convex}\}.$$
\end{theo}

\begin{proof}
By Theorem \ref{existence} it is sufficient to show that
$$\overline{m}_2(p,q)<\overline{m}_1(p,q)=\pi_p/\pi_q.$$
The cases when $q=1$ or $p=+\infty$ follow at once by inequalities \eqref{HP1} and \eqref{Buser}. 
The remaining cases follow by combining the upper estimate for $\la_p(Q)$ and the lower estimate for $\la_q(Q)$ given by Lemma \ref{stimequadrati}; notice that since $p\neq q$, at least one of these two inequalities is strict.
\end{proof}
 
\begin{rem}\label{massimi}
We notice that, by Proposition \ref{Ftouhi}, one readily concludes that if there exists a maximizing sequence $(\O_n)\subset\A^d_{convex}$ such that $|\O_n|=1$ and satisfying $\diam(\O_n)\to+\infty$, then
$$\overline{M}_{d}(p,q)=\overline{M}_{d-1}(p,q).$$
In particular when $d=2$, this argument would prove Conjecture \ref{conj}. On the other hand, if any maximizing sequence $(\O_n)\subset\A^d_{convex}$ with $|\O_n|=1$ is contained (up to translation) in a compact set, arguing as in Theorem \ref{existence} it is easy to conclude that a convex maximizer exists.
\end{rem}
\section{Further remarks and open problems}\label{sfurther}

Several interesting problems and questions about the shape functionals $\F_{p,q}$ are still open; in this section we list some of them.

\bigskip{\bf Problem 1. }In Theorem \ref{theo.supall} we have shown that $M_d(p,q)<+\infty$ when $q>d$; it would be interesting to give a characterization of the quantity $M_d(p,q)$ in these cases. In addition, even if we believe that the value $M_d(p,q)$ is not a maximum, that is it not reached on a domain $\O$, it would be interesting to describe the behavior of maximizing sequences $(\O_n$). It is reasonable to expect that $\O_n$ is made by a domain $\O$ where $n$ points are removed; the locations of these points in $\O$ is an interesting issue: is it true that in the two-dimensional case they are the centers of an hexagonal tiling?

\bigskip{\bf Problem 2. }Proving or disproving the existence of a domain $\O$ minimizing the shape functional $\F_{p,q}$ in the class $\A^d_{all}$ is another very interesting issue. The presence of small holes in a domain $\O$ does not seem to decrease the value of $\F_{p,q}(\O)$, which could be a point in favor of the existence of an optimal domain $\O_{p,q}$.

\bigskip{\bf Problem 3. }In the more restricted class $\A^d_{convex}$ we know that $\overline{M}_d(p,q)$ is always finite. It would be interesting to prove (or disprove) Conjecture \ref{conj} (formulated by Parini in \cite{Pa15}), that is:
$$\overline{M}_d(p,q)=\pi_p/\pi_q\hbox{ and no maximizer exists.}$$
In other words, maximizing sequences are made by thin slabs
$$\O_\eps=A\times(0,\eps)\quad\hbox{with $\eps\to0$ and $A$ a smooth $d-1$ dimensional domain.}$$
At present the problem is open even in the case $d=2$, see also Remark \ref{massimi}.

\bigskip{\bf Problem 4. }Concerning the minimum $\overline{m}_d(p,q)$ of $\F_{p,q}$ in the class $\A^d_{convex}$, establishing if it is attained is an interesting issue. Theorem \ref{main} gives an affirmative answer in the case $d=2$ and $q\le2\le p$; in particular, this happens for $d=2$ and $q=1$, $p=2$, which is the original Cheeger case and, according to some indications by E. Parini \cite{Pa15}, the optimal domain could be in this case a square. This is not yet known.\\
We expect the existence of an optimal domain for every dimension $d$ and every $p,q$ and, as stated in Theorem \ref{existence}, this would follow once the strict monotonicity of $\overline{m}_d(p,q)$ with respect to the dimension $d$ is proved. At present however, a general proof of this strict monotonicity is missing.

\appendix

\section{}
\label{sapp}

We devote this appendix to briefly describe the classical strategy of Cioranescu-Murat (see \cite{ciomur}) which can be used to prove Theorem \ref{theo.supall}, see Remark \ref{cioranescuproof}. These results are well known, but in the case $p\ne2$ it is not easy to find precise references, hence we add them for the sake of completeness and for reader's convenience. We limit ourselves to prove only what we need in the paper, pointing out that the following results can be obtained in a more general framework of $\gamma$-convergence (see for instance the monographs \cite{bubu05}, \cite{He} and references therein).

Let $1\le p\le d$, $\O$ be a bounded connected smooth open set, and $\eps>0$. We consider in $\R^d$ ($d\ge2$) the lattice of parallel cubes $P^i_\eps$ of size $2\eps$ and we denote by $x^\eps_i$ the centers. In each cube we consider a tiny ball $B_{r_\eps}(x^\eps_i)$ of radius $r_\eps$, where $r_\eps<\eps$. Finally we set $$
C_{\eps}=\left\{x_i^\eps\ :\ P^i_\eps\Subset\O\right\}.
$$
and 
$$\O_\eps=\O\setminus\bigcup_{\ x\in C_{\eps}} \overline{B}_{r_\eps}(x).$$
Our goal is to determine the behavior of $\la_p(\O_\eps)$ as $\eps\to 0$. This depends on the size of $r_\eps$, and more precisely on the following ratio:
\be\label{cioranescuratio}
a_\eps=\begin{cases}
\eps^{-d}r_\eps^{d-p}&\hbox{if }p<d,\\
\eps^{-d}(-\ln(r_\eps))^{{1-d}}&\hbox{if }p=d.
\end{cases}
\ee

\begin{prop}[super-critical case]
If $a_\eps\to +\infty$ as $\eps\to 0$, then $\la_p(\O_\eps)\to+\infty$.
\end{prop}

\begin{proof}
Given $R>r>0$, we denote by $\mu_{R,r}$ the least eigenvalue of $B_R\setminus\overline{B}_r$ with Dirichlet boundary condition on $\pa B_r$ and Neumann boundary condition on $\pa B_R$, that is:
\be\label{neumann}
\mu_{R,r}=\inf\left\{\frac{\int_{B_R\setminus B_r}|\nabla v|^pdx}{\int_{B_R\setminus B_r}|v|^pdx}\ :\ v\in W^{1,p}(B_R\setminus \overline{B_r}),\ v=0 \hbox{ on }\pa B_r \right\},
\ee
where the condition $v=0$ on $\pa B_r$ is intended in the usual trace sense.
Notice that by exploiting the convexity property of the functional $u\mapsto\int|\nabla u^{1/p}|^p$ (as done in  \cite{BK}, \cite{Be1} and \cite{DS}), we can infer that there exists a unique positive minimizer $v$ for \eqref{neumann} with unitary $L^p$ norm. In particular, being the domain $B_R\setminus B_r$ radial, $v$ is a radially symmetric function in $W^{1,p}(B_R\setminus \overline{B_r})$.

We claim that there exist constants $c\ge1$ and $\xi>0$ (which do not depend on $\eps$) such that
\be\label{stimaconneu}
\la_p(\O_\eps)\ge \xi\mu_{c\eps, r_\eps}.
\ee
Assume \eqref{stimaconneu} to be true, we obtain the thesis by proving that $\mu_{c\eps,r_\eps}\to +\infty$ as $\eps\to0$. 

Indeed, taking for simplicity $c=1$ and using coarea formula, we have that
$$\mu_{\eps,r_\eps}=\inf\left\{\int_{r_\eps}^\eps |u'(t)|^pt^{d-1}dt\right\},$$ 
where the infimum is computed among non negative functions $u \in C^{\infty}(r_\eps,\eps)$ vanishing on $r_\eps$ and satisfying
\be\label{unitary}
\int_{r_\eps}^\eps|u(t)|^{p}t^{d-1}dt=1.
\ee 
If $u$ is admissible, by H\"older inequality,
\be \label{referata}
\int_{r_\eps}^\eps |u'(t)|^pt^{d-1}dt\ge\left(\int_{r_\eps}^\eps|u'(t)|dt\right)^p\left(\int_{r_\eps}^\eps\frac{1}{t^{(d-1)/(p-1)}}dt\right)^{1-p};
\ee
moreover there exists $t_0\in (r_\eps,\eps)$ such that
$$|u(t_0)|=\left(\int_{r_\eps}^\eps	t^{d-1}dt\right)^{-1/p},$$
since otherwise
$$
0\le u(t)<\left(\int_{r_\eps}^\eps	t^{d-1}dt\right)^{-1/p} \hbox{for every } t\in (r_\eps,\eps),
$$
would imply
$\int_{r_\eps}^\eps |u(t)|^pt^{d-1}dt<1$, in contradiction with \eqref{unitary}.
Hence, using the fact that $u(r_{\eps})=0$, we have
$$\int_{r_\eps}^\eps |u'(t)|dt\ge \left |\int_{r_\eps}^{t_0}u'(t)dt\right|\ge\left(\int_{r_\eps}^\eps t^{d-1}dt\right)^{-1/p}.$$
Finally, the latter inequality combined with \eqref{referata} implies
$$
\int_{r_\eps}^\eps |u'(t)|^pt^{d-1}dt\ge \left(\int_{r_\eps}^\eps t^{d-1}dt\right)^{-1}\left(\int_{r_\eps}^\eps \frac{1}{t^{(d-1)/(p-1)}}dt\right)^{1-p}. 
$$
In the case $p<d$ (the case $p=d$ being similar), computing the right-hand side in the previous inequality we obtain
$$
\mu_{\eps, r_\eps}\ge d\left(\frac{1}{\eps^d-r_{\eps}^d}\right) \left(\frac{d-p}{p-1}\right)^{p-1}\left(\frac{1}{r_{\eps}^{(p-d)/(p-1)}-\eps^{(p-d)/(p-1)}}\right)^{p-1}.
$$
Taking \eqref{cioranescuratio} into account, it is easy to verify that the right hand side of the previous inequality tends to $+\infty$ as $\eps\to0$.

To conclude let us prove \eqref{stimaconneu}. We notice that there exists $c>1$, which does not depend on $\eps$, such that for every $\eps$ small enough the family of balls 
$$
\mathcal{G}^\eps=\left\{B_{c\eps}(x)\ :\ x\in C_\eps\right\},
$$
covers $\O$. Moreover there exists $N\in\N$, which again does not depend on $\eps$, such that we can split $\mathcal{G}^\eps$ into $N$ sub-families $\mathcal{G}_1^\eps,\dots \mathcal{G}_N^\eps$ made up of disjoint balls. This latter assertion can be easily proved once noticed that any ball in $\mathcal{G}^\eps$ can intersect only a bounded number of different balls in $\mathcal{G}^\eps$, and such a bound does not depend on $\eps$. Indeed suppose that $B_{c\eps}(\bar x)\in\mathcal{G}^\eps$ intersects $B_{c\eps}(x_1),\dots,B_{c\eps}(x_m)\in\mathcal{G}^\eps$, then we have also
$$
 \bigcup_{i=1}^m B_\eps(x_i) \Subset B_{3c\eps}(\bar x),
$$
in particular, taking the measures of both sets, we get $m\le(3c)^d$. Therefore, it is sufficient to take $N=[(3c)^d]+1.$

Now, let $u\in C_c^{\infty}(\O_\eps)$ and extend $u$ by zero outside $\O_\eps$. We have
\[\begin{split}
&N\int_{\O}|\nabla u|^pdx\ge \sum_{i=1}^N\sum_{B\in\mathcal{G}_i^\eps}\int_{B}|\nabla u|^pdx\ge \mu_{c\eps,r_{\eps}}\sum_{i=1}^N\sum_{B\in\mathcal{G}_i^\eps}\int_{B}|u|^pdx\geq \mu_{c\eps,r_\eps}\int_{\O}|u|^pdx.
\end{split}\]
Thus, by the arbitrariness of $u$ we obtain \eqref{stimaconneu} with $\xi=N^{-1}$.
\end{proof}

\begin{prop}[sub-critical case]
If $a_\eps\to 0$ as $\eps\to 0$, then $\la_p(\O_\eps)\to\la_p(\O)$.
\end{prop}

\begin{proof}
First we notice that by monotonicity we have
$$\la_p(\O_\eps)\ge \la_p(\O).$$
Hence it is enough to prove that
\be\label{moscoup}
\limsup_{\eps\to 0}\la_p(\O_\eps)\le \la_p(\O).
\ee
Let $v_\eps$ be a competitor for $\ccap_p(\overline{B}_{r_\eps};B_\eps)$ chosen in such a way that:
$$
v_\eps\in C^\infty_c(B_\eps),\quad 0\le v_\eps\le1,\quad v_\eps=1\hbox{ on }B_{r_\eps},\quad\|\nabla v_\eps\|_p\le\ccap_p(\overline{B}_{r_\eps};B_\eps)+o(\eps^d).
$$ 
 We define $V_\eps$ in $\O$ to be
$$V_\eps(x)=
\begin{cases}
1-v_\eps(x-x_i),& \hbox{in }B^i_\eps(x_i) \hbox{ if } x_i\in C_{\eps}\\
1, &\hbox{in } \Omega\setminus \bigcup_{x\in C_{\eps}} B^i_\eps(x),
\end{cases}$$
and we denote by $n(\eps)\in\N$ be the number of cubes $P^i_\eps$ such that $P_i^\eps\Subset\O$.
We have
$$
\|\nabla V_\eps\|_{L^p(\O)}\le n(\eps)\|\nabla v_\eps\|_{L^p(B_\eps)}\le (2\eps)^{-d}|\O|\|\nabla v_\eps\|_{L^p(B_\eps)}\approx (2\eps)^{-d}|\O|\ccap_p(\overline{B}_{r_\eps};B_\eps) .
$$
Since $r_\eps\to0$, the latter implies 
$$
\|\nabla V_\eps\|_{L^p(\O)}\to0
$$
(see Section 2.2.4 of \cite{Maz} for the precise value of $\ccap_p(\overline{B}_{r_\eps};B_\eps)$).
This means that $V_\eps$ weakly converge in $W^{1,p}(\O)$ to some constant $c\in\R$. Moreover, since $V_\eps=1$ on $\partial\O$, we can infer that $c=1$.

Now, let $u\in C^{\infty}_c(\O)$, and consider $u_\eps=V_\eps u$. We have $u_\eps\in W^{1,p}_0(\O_\eps)$ and $u_\eps\to u$ strongly in $W_0^{1,p}(\O)$. In particular
$$
\frac{\int_\O|\nabla u|^pdx}{\int_\O|u|^pdx}=\lim_{\eps\to 0}\frac{\int_{\O_\eps}|\nabla u_\eps|^pdx}{\int_{\O_\eps}|u_\eps|^pdx}\ge \limsup_{\eps\to 0}\la_p(\O_\eps).
$$
Since $u$ is arbitrary we get \eqref{moscoup}.
\end{proof}

\bigskip

\noindent{\bf Acknowledgments.} We wish to thank Lorenzo Brasco for the useful discussions on the subject. The work of GB is part of the project 2017TEXA3H {\it``Gradient flows, Optimal Transport and Metric Measure Structures''} funded by the Italian Ministry of Research and University. The authors are member of the Gruppo Nazionale per l'Analisi Matematica, la Probabilit\`a e le loro Applicazioni (GNAMPA) of the Istituto Nazionale di Alta Matematica (INdAM).

\bigskip
{\small\noindent
Luca Briani:
Dipartimento di Matematica,
Universit\`a di Pisa\\
Largo B. Pontecorvo 5,
56127 Pisa - ITALY\\
{\tt luca.briani@phd.unipi.it}

\bigskip\noindent
Giuseppe Buttazzo:
Dipartimento di Matematica,
Universit\`a di Pisa\\
Largo B. Pontecorvo 5,
56127 Pisa - ITALY\\
{\tt giuseppe.buttazzo@dm.unipi.it}\\
{\tt http://www.dm.unipi.it/pages/buttazzo/}

\bigskip\noindent
Francesca Prinari:
Dipartimento di Scienze Agrarie, Alimentari e Agro-ambientali,\\
Universit\`a di Pisa\\
Via del Borghetto 80,
 56124 Pisa - ITALY\\
{\tt francesca.prinari@unipi.it}

\end{document}